\documentclass[reqno]{amsart}
\usepackage{amssymb}
\usepackage{eucal}
\usepackage{amsmath}
\usepackage{amscd}
\usepackage[dvips]{color}
\usepackage{multicol}
\usepackage[all]{xy}           
\usepackage{graphicx}
\usepackage{color}
\usepackage{colordvi}
\usepackage{xspace}
\usepackage{bookmark}

\usepackage[active]{srcltx} 

\usepackage{hyperref}
\usepackage{lipsum}

\newcommand{\delete}[1]{}

\newtheorem{theorem}{Theorem}[section]
\newtheorem{definition}[theorem]{Definition}
\newtheorem{lemma}[theorem]{Lemma}

\newtheorem{prop-def}[theorem]{Proposition-Definition}

\newtheorem{remark}[theorem]{Remark}

\newtheorem{example}[theorem]{Example}

\makeatletter
\g@addto@macro{\endabstract}{\@setabstract}
\newcommand{\authorfootnotes}{\renewcommand\thefootnote{\@fnsymbol\c@footnote}}%
\makeatother
\begin{document}
\begin{center}
  \LARGE
\textbf{Biderivations of the twisted Heisenberg-Virasoro algebra and their applications }

  \normalsize
  \authorfootnotes
Xiaomin Tang  \footnote{Corresponding author: {\it X. Tang. Email:} x.m.tang@163.com}
,Xiaotong Li
\par \bigskip
   \textsuperscript{1}Department of Mathematics, Heilongjiang University,
Harbin, 150080, P. R. China   \par

\end{center}


\begin{abstract}

In this paper, the biderivations without the skew-symmetric condition of the twisted Heisenberg-Virasoro algebra are presented. We find some non-inner and non-skew-symmetric biderivations. As applications, the characterizations of the forms of linear commuting maps and the commutative post-Lie algebra structures on the twisted Heisenberg-Virasoro algebra are given. It also is proved that every biderivation of the graded twisted Heisenberg-Virasoro left-symmetric algebra is trivial.

\vspace{2mm}



\end{abstract}

\setcounter{section}{0}
{\ }

 \baselineskip=20pt

\section{Introduction}\label{intro}

Derivations and generalized derivations are very important subjects in algebra.
Now let us recall the definitions of the derivation and biderivation of an algebra as follows. Suppose that $L$ be a vector space equipped with a binary operation $(x, y)\mapsto x \circ y$ for all $x,y\in L$. As usually, $(L, \circ)$ is said to be an algebra (not necessarily is an associative algebra, for example, as the Lie algebra, left-symmetric algebra, etc.).
\begin{itemize}
\item A linear map $\phi: L\rightarrow L$ is called a derivation if it satisfies
\begin{eqnarray*}
\phi(x\circ y)=\phi(x)\circ y +x \circ \phi(y)
\end{eqnarray*}
for all $x, y\in L$. If $L$ is a Lie algebra, then we denote by $x\circ y=[x,y]$ for all $x,y\in L$. In this case, for $x\in L$, it is easy to see that $\phi_x: L\rightarrow L, y\mapsto {\rm ad} x(y)=[x,y], $ for all $y\in L$ is a derivation of $L$, which is called an inner derivation.  Denote by ${\rm{Der}} (L)$ and by ${\rm{Inn}} (L)$ the space of derivations and the space of inner derivations of $L$ respectively.
\end{itemize}

\begin{itemize}
\item A bilinear map $f : L\times L \rightarrow L$ is called a
biderivation of $L$ if it is a derivation with respect to both components. Namely, for each $x\in L$, both linear maps $\phi_x$ and $\psi_x$ form $L$ into itself given by $\phi_x=f(x, \cdot)$ and $\psi_x=f(\cdot, x)$ are derivations of $L$, i.e.,
\begin{eqnarray}
f(x\circ y,z)=x \circ f(y,z)+ f(x,z)\circ y, \label{2der}\\
f(x,y \circ z)=f(x,y)\circ z+y \circ f(x,z)  \label{1der}
\end{eqnarray}
for all $x, y, z\in L$. Denote by $B(L)$ the set of all biderivations of $L$. For a Lie algebra $L$
and $\lambda\in \mathbb{C}$, it is easy to verify that the bilinear map $f: L\times L\rightarrow L$ given by $f(x,y)=\lambda [x,y]$ for all $y\in L$
is a biderivation of $L$. Such biderivation is said to be inner. Recall that $f$ is skew-symmetric if $f(x,y)=-f(y,x)$ for all $x,y\in L$.
\end{itemize}

 In recent years, many authors put so much effort into the problems of biderivations \cite{Bre1995,Chen2016,c2017,Du2013,Gho2013,Hanw,tang2016,WD1,WD3,WD2}. In \cite{Bre1995}, Bre\v{s}ar et al. showed that all biderivations on commutative prime rings are inner biderivations, and determine the biderivations of semiprime rings. The notion of biderivations of Lie algebras was introduced in \cite{WD3}.
  In recent years, many authors put so much effort into this problem \cite{Bre1995,Chen2016,c2017,Du2013,Gho2013,Hanw,tang2016,vela,WD1,WD3}. In this paper, we shall study the biderivation of the twisted Heisenberg-Virasoro algebra. The twisted Heisenberg-Virasoro algebra is one of the most important Lie algebras both in mathematics and in mathematical physics. The structure and representation theories of the twisted Heisenberg-Virasoro algebra and its various extended Lie algebras have been extensively investigated (see, e.g.,\cite{Billig,Chenhj,Chenhj2,Shenr,Wangq}). Now we give the definition of the twisted Heisenberg-Virasoro algebra and an example of its biderivation as follows.

\begin{definition}
The twisted Heisenberg-Virasoro algebra is the Lie algebra $(\mathcal{H}, [, ])$ which has $\mathbb{C}$-basis $$\{L_n, I_n, C_1,C_{2}, C_3| n\in \mathbb{Z} \}$$
and the following Lie brackets, for all $m,n\in \mathbb{Z}$,
\begin{eqnarray*}
&&[L_n,L_m ]=(n-m)L_{m+n}+\frac{n^3-n}{12}\delta_{n,-m}C_1,\\
&& [L_n,I_m ]=-mI_{m+n} -{(n^2+n)}\delta_{n,-m}C_2, \\
&&[I_n,I_m ]=n\delta_{n,-m}C_3, \ [\mathcal{H},C_1 ]=[\mathcal{H},C_2 ]=[\mathcal{H},C_3 ]=0.
\end{eqnarray*}
\end{definition}

\begin{example}\label{example}
Let $f: \mathcal{H}\times \mathcal{H} \rightarrow \mathcal{H}$ be a bilinear map determined by $f(L_m, L_n)=I_{m+n}$, $f(L_m, I_n)=f(I_n, L_m)=f(I_n,I_m)=0$ and $f(x,C_i)=f(C_i,y)=0, i=1,2,3$ for all $m,n\in \mathbb{Z}$ and $x,y\in \mathcal{H}$. Then it is easy to verify that $f$ is a biderivation of $\mathcal{H}$. Note that it is non-inner and non-skew-symmetric.
\end{example}

Throughout the paper, the symbols $\mathbb{C}$ and $\mathbb{Z}$ represent for the sets of complex numbers and integers, respectively. Recall that a Lie algebra $L$ is said to be perfect if $L=[L,L]$. Note that the twisted Heisenberg-Virasoro algebra is perfect. Denote by $Z(L)$ the center of a Lie algebra $L$. Obviously, $Z(\mathcal{H})={\rm Span}\{I_0, C_1, C_2, C_3\}$. We denote a subset of $Z(\mathcal{H})$ by $Z^\prime(\mathcal{H})={\rm Span}\{ C_1, C_2, C_3\}$.

 Our results can be briefly summarized as follows: In Section \ref{sec2}, we give some lemmas which will be used to our proof. In Section \ref{sec3}, we characterize the biderivations without the skew-symmetric condition of the twisted Heisenberg-Virasoro algebra. In Section \ref{sec4}, we give three applications of biderivation of the twisted Heisenberg-Virasoro algebra, i.e., the characterizations of the forms of linear commuting maps and the commutative post-Lie algebra structures on the twisted Heisenberg-Virasoro algebra are given. As a simple corollary, we prove that every biderivation of the graded twisted Heisenberg-Virasoro left-symmetric algebra is trivial.

\section{Preliminary results}\label{sec2}

\begin{lemma}\label{refe} Suppose that $k_i^{(n)},h_i^{(m)}\in\mathbb{C},i,m,n\in\mathbb{Z}$ such that
\begin{eqnarray}(i-m)k_i^{(n)}=(2n-m-i)h_{m-n+i}^{(m)}\label{aaa4}
\end{eqnarray}
for all $m,n,i\in\mathbb{Z} $. Then there is $\lambda\in\mathbb{C}$ such that $k_i^{(m)}=h_i^{(m)}=\delta_{m,i}\lambda$ for all $ m,i \in\mathbb{Z}$.
\end{lemma}
\begin{proof}
For any $ m,n$ with $m\neq n$, by taking $i=2n-m$,$m$ and $i=n$ in (\ref{aaa4}) respectively, we have
\begin{eqnarray}
k_{2n-m}^{(n)}=h_{2m-n}^{(m)}=0,k_{n}^{(n)}=h_{m}^{(m)},\forall m,n\in\mathbb{Z} \ {\rm with}\  m\neq n.\label{aaa5}
\end{eqnarray}

Let $m,n$ run all integers with $ m\neq n$, then by (\ref{aaa5}) we see that $k_{i}^{(n)}=0,h_{j}^{(m)}=0$ for all $i\neq n,\ j\neq m$ and $k_{n}^{(n)}=h_{n}^{(n)}=h_{0}^{(0)}$
for all $n\in\mathbb{Z}$. By letting $h_{0}^{(0)}=\lambda$, the proof is completed.
\end{proof}
\begin{lemma}\label{refer}
 Suppose that $t_i^{(n)},g_i^{(m)}\in \mathbb{C},\ i,m,n \in \mathbb {Z} $ and $\rho_2,\rho_3, \theta_2,\theta_3$
 are linear complex valued functions on $ \widetilde{\mathcal{H}}$ satisfy
 \begin{eqnarray} it_i^{(n)}=-(m-n+i)g_{m-n+i}^{(m)}, \ \text{for}\  i\neq 0,n-m \label{aaa6}\\
 (m-1)\rho_2(L_n)+m\rho_3(L_n)=-(m-n)g_{m-n}^{(m)},\label{aaa7} \\
 (n-1)\theta_2(L_m)+n\theta_3(L_m)=(n-m)t_{n-m}^{(n)} \label{aaa8}
 \end{eqnarray}
 for all $i,m,n \in \mathbb {Z}$. Then there is a set of complex numbers
 $\Omega=\{\mu_k\in \mathbb{C}| k\in \mathbb{Z}\}$ such that
 \begin{eqnarray}
 t_{n+k}^{(n)}=-g_{n+k}^{(n)}=\frac{\mu_k}{n+k} ,\forall k\in\mathbb{Z}\backslash \{-n\}, \label{aaa9}
  \end{eqnarray}
  \begin{eqnarray}
 \rho_2(L_n)=\theta_2(L_n)=-\mu_{-n}, \label{aa10}
  \end{eqnarray}
    \begin{eqnarray}
 \rho_3(L_n)=\theta_3(L_n)=\mu_{-n}. \label{aa11}
  \end{eqnarray}
\end{lemma}
\begin{proof}
By taking $i=n+k$ in (\ref{aaa6}) with $k\neq-m$, $k\neq-n$, we have
 \begin{eqnarray}
(n+k)t_{n+k}^{(n)}=-(m+k)g_{m+k}^{(m)}. \label{aa12}
  \end{eqnarray}
  Let $m=n=1-k$ in (\ref{aa12}), one has $t_{1}^{(1-k)}=-g_{1}^{(1-k)}$. Denote $\mu_k=t_{1}^{(1-k)}$.
  By using (\ref{aa12}) with $m=1-k$, we deduce that $t_{n+k}^{(n)}=\frac{\mu_k}{n+k} $  for all $ k\in\mathbb{Z}\backslash \{-n\}$.
  Similarly, by using (\ref{aa12}) with $n=1-k$ we
  see that $g_{m+k}^{(m)}=-\frac{\mu_k}{m+k} $  for all $ k\in\mathbb{Z}\backslash \{-m\}$. This proves that (\ref{aaa9}) holds. Notice that (\ref{aaa9}) implies $(m-n)g_{m-n}^{(m)}=-\mu_{-n}.$ This, together with (\ref{aaa7}), yields that
 \begin{eqnarray}
(m-1)(\rho_2(L_n)+\rho_3(L_n))+\rho_3(L_n)=\mu_{-n}, \forall n\in \mathbb{Z}. \label{aa13}
\end{eqnarray}
This indicates that $\rho_3(L_n)=-\rho_2(L_n)=\mu_{-n}$. Similarly, notice that (\ref{aaa9}) implies that $(n-m)t_{n-m}^n=\mu_{-m}$ and then by (\ref{aaa8}) we deduce that  $-\theta_2(L_m)=\theta_3(L_m)=\mu_{-m}$. The proof is completed.
\end{proof}

\begin{lemma}\label{refere}
Suppose that $s_i^{(n)},l_i^{(n)}\in\mathbb{C}$ for all $i,n\in\mathbb{Z}$.

(i) If $(m-i)s_i^{(n)}=0$ for all $i,m,n \in \mathbb {Z}$, then $s_i^{(n)}=0$ for all $i,n \in \mathbb {Z}.$

(ii) If $j l_j^{(m)}=0$ for all $j\neq m, m-n$, then $l_j^{(m)}=0$ for all $j\neq 0$.
\end{lemma}

\begin{proof}
The proof is similar to Lemma \ref{refe}.
\end{proof}

\begin{lemma}\label{fxc=4}
Let $L$ be a perfect Lie algebra and  $f$ be a biderivation of $L$.
If $\alpha \in Z(L)$, then $f(x, \alpha)=f(\alpha,x)=0$ for all $x\in L$.
\end{lemma}

\begin{proof}
For each $x\in L$, since $\phi_x=f(x, \cdot)$ is a derivation of $L$, so $0=\phi_x ([\alpha,y])=[\phi_x(\alpha), y]+[\alpha,\phi_x(y)]=[\phi_x(\alpha), y]$
for all $y\in L$. In the other words, $\phi_x(\alpha)=f(x, \alpha)\in Z(L)$. Similarly, we have $f(\alpha,x)\in Z(L)$. Now from (\ref{2der})
 and (\ref{1der}) one has $f([y,z],\alpha)=[y,f(z,\alpha)]+[f(y,\alpha),z]=0$ and $f(\alpha,[y,z])=[f(\alpha,y),z]+[y,f(\alpha,z)]=0$
  for all $y,z\in L$. Thanks to $L=[L,L]={\rm Span}\{[y,z]|y,z\in L\}$, forcing $f(x, \alpha)=f(\alpha,x)=0$ for all $x\in L$.
\end{proof}

 \begin{lemma}\label{fxc=5}
  Let $L$ be a perfect Lie algebra and $Z^\prime (L)\subset Z(L)$. Then the linear map $\pi: B(L)\rightarrow B(L/Z^\prime (L))$ defined by $\pi(f)(x+Z^\prime (L),y+Z^\prime (L))=f(x,y)+Z^\prime (L)$ is injective. Furthermore, $\pi$ is bijective if and only if any biderivation of $L/Z^\prime (L))$ can be extended to a biderivation of $L$.\end{lemma}
\begin{proof} From Lemma \ref{fxc=4}, $\pi$ is well-defined. Now for any $f\in \ker \pi$, i.e., $f(x,y)\in Z^\prime (L),\forall x,y\in L$. Then from
$f([x,y],z)=[f(x,z),y]+[x,f(y,z)]=0$ and $L$ is perfect, we have $f=0$. The second part of the lemma is trivial. \end{proof}

For a Lie algebra $\widetilde{L}$, let $L=\widetilde{L}\oplus Z^\prime (L)$ be a central extension of $\widetilde{L}$. Then it is well known that
$\widetilde{L}\cong L/ Z^\prime(L)$. Furthermore, Lemmas \ref{fxc=4} and \ref{fxc=5} tell us that
if $L$ is perfect then we only need to characterize the biderivation of $\widetilde{L}$.

\section{ Biderivation of the twisted Heisenberg-Virasoro algebra}\label{sec3}

We fist give a class of biderivations of the twisted Heisenberg-Virasoro algebra $\mathcal{H}$ as follows.
\begin{definition} \label{taa}
Let $\Omega=\{\mu_k\in \mathbb{C}| k\in \mathbb{Z}\}$ be a set satisfying $|\Omega\vert < +\infty$, $i.e.,\Omega$ only contains finitely many nonzero numbers. For such $\Omega$, we define a bilinear map $r_\Omega: \mathcal{H}\times \mathcal{H}\rightarrow \mathcal{H}$  given by
\begin{eqnarray}
 r_\Omega(L_m,L_n)=\sum_{k \in \mathbb{Z}}\mu_k I_{m+n+k} \label{fff3}
\end{eqnarray}
for all $m,n\in\mathbb{Z}$ and $ r_\Omega(x,y)=0$ if either of $x,y$ is contained in $\{C_1, C_2, C_3, I_n| n\in \mathbb{Z}\}$.
\end{definition}

It is easy to verify that the above $r_\Omega$ is a biderivation of twisted Heisenberg-Virasoro algebra. Note that $r_\Omega$ is symmetric and non-inner. The case in which $\Omega=\{\mu_0=1, \mu_i=0|i\in \mathbb{Z}\setminus \{0\}\}$  gives just Example \ref{example}. Our main result is following.

\begin{theorem}\label{posttheo}
$f$ is a biderivation  of $\mathcal{H}$ if and only if there are $\lambda \in \mathbb{C}$ and a finite set $\Omega=\{\mu_k\in \mathbb{C}| k\in \mathbb{Z}\}$ such that
\begin{eqnarray}
f(x,y)=\lambda[x,y]+r_\Omega(x,y) \label{fff4}
\end{eqnarray}
for all $x,y\in \mathcal{H}$, where $r_\Omega$ is given by Definition \ref{taa}.
\end{theorem}

The proof of Theorem \ref{posttheo} will be completed by several lemmas as follows. Let $\widetilde{\mathcal{H}}=\mathcal{H} /Z^{\prime}(\mathcal{H})$ where $Z^\prime(\mathcal{H})={\rm Span}\{ C_1, C_2, C_3\}$. Then $\widetilde{\mathcal{H}}$ just is the Lie algebra $W(0,0)$ \cite{Shoulan},  which has the basis  $\{L_n,I_n \lvert n\in\mathbb{Z}\}$ and the Lie bracket
\begin{eqnarray}
\ [L_m,L_n ]=(m-n)L_{m+n},
\ [L_m,I_n ]=-nI_{m+n}, \ [I_m,I_n ]=0.\label{aaa1}
\end{eqnarray}
\begin{lemma}\label{fxc=0} \cite{Shoulan}
${\rm{Der}} \widetilde{\mathcal{H}}={\rm{I{nn}}} \widetilde{\mathcal{H}} \oplus \mathbb{C} D_1\oplus \mathbb{C} D_2^{0,0}\oplus \mathbb{C} D_3$,
where $D_1,D_2^{0,0},D_3$ are outer derivations defined by $D_1(L_m)=0$, $D_1(I_m)=I_m$,
$D_2^{0,0}(L_m)=(m-1)I_m,$ $D_2^{0,0}(I_m)=0,$
$D_3(L_m)=mI_m,$ $D_3(I_m)=0$ for all $m\in \mathbb{Z}$.
\end{lemma}

\begin{lemma}\label{referee}
Suppose that $f$ is a biderivation of $ \widetilde{\mathcal{H}},$  Then there are linear maps $\phi$ and $\psi$ from $ \widetilde{\mathcal{H}}$ into itself such that
\begin{eqnarray}
f(x,y)&=&\rho_1(x) D_1(y)+\rho_2(x)D_2^{0,0}(y)+\rho_3(x)D_3(y)+[\phi(x), y],\nonumber\\
&=&\theta_1(y) D_1(x)+\theta_2(y)D_2^{0,0}(x)+\theta_3(y)D_3(x)+[x, \psi(y)] \label{abcd1} \label{aaa2}
\end{eqnarray}
for all $x,y\in \widetilde{\mathcal{H}}$, where $\rho_1, \rho_2,\rho_3$ and $\theta_1,\theta_2,\theta_3$ are linear complex value function on $\widetilde{\mathcal{H}}$ and $D_1,D_2^{0,0},D_3$ is given by Lemma \ref{fxc=0}
\end{lemma}

\begin{proof}
For the biderivation $f$ on $ \widetilde{\mathcal{H}}$ and a fixed element $x\in \widetilde{\mathcal{H}}$, we define a map $\phi_x:  \widetilde{\mathcal{H}}\rightarrow  \widetilde{\mathcal{H}}$ is given by $\phi_x(y)=f(x,y)$. Then we know by (\ref{1der}) that $\phi_x$ is a derivation of $ \widetilde{\mathcal{H}}$. By Lemma \ref{fxc=0}, there are three complex value function $\rho_1, \rho_2, \rho_3$ on $\widetilde{\mathcal{H}}$ and a linear map $\phi$ from $\widetilde{\mathcal{H}}$ into itself such that $\phi_x =\rho_1(x) D_1+\rho_2(x) D_2^{0,0} +\rho_3 (x) D_3+{\rm ad}\phi(x)$.
Namely, $f(x,y)=\rho_1(x) D_1(y)+\rho_2(x)D_2^{0,0}(y)+\rho_3(x) D_3(y)+[\phi(x), y]$. Due to $f$ is bilinear, $\rho_1, \rho_2, \rho_3$ are linear. Similarly, if we define a map $\psi_z$ from $\widetilde{\mathcal{H}}$ into itself is given by $\widetilde{\mathcal{H}}$ for all $y\in \widetilde{\mathcal{H}} $, then it can obtained three linear complex valued functions $\theta_1,\theta_2,\theta_3$ on $\widetilde{\mathcal{H}}$ and a map $\psi$ from $\widetilde{\mathcal{H}}$ into itself such that
\begin{eqnarray*}f(x,y)&=&\theta_1(y) D_1(x)+\theta_2(y)D_2^{0,0}(x)+\theta_3(y)D_3(x)+{\rm ad}(-\psi (y))(x)\nonumber\\
&=&\theta_1(y) D_1(x)+\theta_2(y)D_2^{0,0}(x)+\theta_3(y)D_3(x)+[x, \psi(y)].\label{aaa3}
\end{eqnarray*}
 The proof is completed.
\end{proof}

\begin{lemma}\label{ref} Let $f$ be a biderivation of $\widetilde{\mathcal{H}}$ and $\phi, \psi, \rho_i, \theta_i, i=1,2,3$ be given by Lemma \ref{referee}.
Then the following equations hold.
\begin{eqnarray}
f(L_n,L_m)&=&\rho_2(L_n)(m-1)I_m+\rho_3(L_n)mI_m+[\phi (L_n),L_m] \nonumber\\
&=& \theta_2(L_m)(n-1)I_n+\theta_3(L_m)nI_n+[L_n,\psi(L_m)]. \label{aa14}
\end{eqnarray}
\begin{eqnarray}
f(L_n,I_m)&=&\rho_1(L_n)I_m+[\phi (L_n),I_m]\nonumber\\
&=& \theta_2(I_m)(n-1)I_n+\theta_3(I_m)nI_n+[L_n,\psi(I_m)]. \label{aa15}
\end{eqnarray}
\begin{eqnarray}
f(I_n,L_m)&=&\rho_2(I_n)(m-1)I_m+\rho_3(I_n)mI_m+[\phi (I_n),L_m]\nonumber\\
&=& \theta_1(L_m)I_n+[I_n,\psi(L_m)]. \label{aa16}
\end{eqnarray}
\begin{eqnarray}
f(I_n,I_m)&=&\rho_1(I_n)I_m+[\phi (I_n),I_m]=\theta_1(I_m)I_n+[I_n,\psi(I_m)]. \label{aa17}
\end{eqnarray}
\end{lemma}

\begin{proof}
It will follow by Lemmas \ref{fxc=0} and \ref{referee}.
\end{proof}

\begin{lemma}\label{fxc=3}
Let $f$ be a biderivation of $\widetilde{\mathcal{H}}$ and $\phi, \psi, \rho_i, \theta_i, i=1,2,3$ be given by Lemma \ref{referee}.
Then there are $\lambda \in\mathbb{C}$ and a finite set
 $\Omega=\{\mu_k\in \mathbb{C}| k\in \mathbb{Z}\}$ such that (\ref{aa10}), (\ref{aa11}) and
\begin{eqnarray}
\phi(L_n)=\lambda L_n +\sum_{k\in \mathbb{Z}\setminus{\{-n\}}} \frac{\mu_k}{n+k}I_{n+k}+t_{0}^{(n)}I_0,\label{aa18}\\
\psi(L_n)=\lambda L_n +\sum_{k\in \mathbb{Z}\setminus{\{-n\}}} \frac{-\mu_k}{n+k}I_{n+k}+g_{0}^{(n)}I_0,\label{aa19}
\end{eqnarray}
for all $n\in\mathbb{Z}$, where $t_0^{(n)}$ and $g_0^{(n)}$ are complex numbers related to $n$. And then, we have
\begin{eqnarray}
f(L_n, L_m)=\lambda [L_n,L_m] +r_\Omega(L_n,L_m) \label{aa1888}
\end{eqnarray}
where $r_\Omega(L_n,L_m)=\sum_{k\in \mathbb{Z}} \mu_k I_{m+n+k}$ is given by Definition \ref{taa}.
\end{lemma}
\begin{proof} For any fixed $n\in\mathbb{Z}$, suppose that
\begin{eqnarray}
\phi (L_n)=\sum_{i\in \mathbb{Z}}k_i^{(n)}L_i+\sum_{i\in \mathbb{Z}}t_i^{(n)}I_i, \label{aa20} \\
\psi(L_n)=\sum_{j\in \mathbb{Z}}h_j^{(n)}L_j+\sum_{j\in \mathbb{Z}}g_j^{(n)}I_j \label{bbb1}
\end{eqnarray}
where $k_i^{(n)},t_i^{(n)},h_i^{(n)},g_i^{(n)}\in \mathbb{C}$ for every $i\in \mathbb{Z}$. With (\ref{aa20}) and (\ref{bbb1}), direct computations show that
\begin{eqnarray}
[\phi (L_n),L_m]=\sum_{i\in \mathbb{Z}}(i-m)k_i^{(n)}L_{i+m}+\sum_{i\in \mathbb{Z}}i t_i^{(n)}I_{i+m}, \label{bbb2}
\end{eqnarray}
\begin{eqnarray}
[L_n,\psi(L_m)]&=&\sum_{j\in \mathbb{Z}}h_j^{(m)}(n-j)L_{j+n}-\sum_{j\in \mathbb{Z}}j g_j^{(m)}I_{j+n} \nonumber\\
 &=& \sum_{i\in \mathbb{Z}}(2n-m-i)h_{i+m-n}^{(m)}L_{i+m}-\sum_{i\in \mathbb{Z}}(i+m-n)g_{i+m-n}^{(m)}I_{i+m}. \label{bbb3}
\end{eqnarray}
Note that (\ref{aa14}) together with (\ref{bbb2}), (\ref{bbb3}) tells us that all (\ref{aaa4}), (\ref{aaa6}), (\ref{aaa7}) and (\ref{aaa8}) hold.
 Therefore, by Lemmas \ref{refe} and \ref{refer}, we know that there are $\lambda \in\mathbb{C}$ and a finite set
 $\Omega=\{\mu_k\in \mathbb{C}| k\in \mathbb{Z}\}$ such that (\ref{aa10}), (\ref{aa11}), (\ref{aa18}) and (\ref{aa19}) are established. And then, by (\ref{aa14}),  (\ref{aa10}), (\ref{aa11}) and (\ref{aa18}) one can obtain that
 \begin{eqnarray*}
f(L_n,L_m)&=&-(m-1)\mu_{-n}I_m+m\mu_{-n}I_m+\lambda [L_n,L_m]+\sum_{k\in \mathbb{Z}\setminus{\{-n\}}} \frac{\mu_k}{n+k}[I_{n+k},L_m]\\
&=& \mu_{-n}I_m+\lambda [L_n,L_m]+ \sum_{k\in \mathbb{Z}\setminus{\{-n\}}} \mu_k I_{m+n+k}\\
&=&  \lambda [L_n,L_m]+ \sum_{k\in \mathbb{Z}} \mu_k I_{m+n+k},
\end{eqnarray*}
which gives (\ref{aa1888}). The proof is completed.
\end{proof}

\begin{lemma}\label{fxc=2}
Let $f$ be a biderivation of $\widetilde{\mathcal{H}}$,  $\lambda$ be given by Lemma \ref{fxc=3} and $\phi, \psi, \rho_i, \theta_i, i=1,2,3$ be given by Lemma \ref{referee}. Then
$$
\phi(I_n)=\lambda I_n +q_{0}^{(n)}I_0,\
\psi(I_n)=\lambda I_n +l_{0}^{(n)}I_0
$$
and
$\rho_1(L_n)=\theta_1(L_n)=\rho_i(I_n)=\theta_i(I_n)=0, i=1,2,3$,
for all $n\in\mathbb{Z}$ where $t_0^{(n)}$ and $q_{0}^{(n)}$ are complex numbers related to $n$.
\end{lemma}
\begin{proof} For any fixed $n\in\mathbb{Z}$, we suppose that
\begin{eqnarray}
\phi (I_n)=\sum_{i\in \mathbb{Z}}p_i^{(n)}L_i+\sum_{i\in \mathbb{Z}}q_i^{(n)}I_i, \label{dd10}\\
\psi(I_n)=\sum_{j\in \mathbb{Z}}s_j^{(n)}L_j+\sum_{j\in \mathbb{Z}}l_j^{(n)}I_j \label{eee1}
\end{eqnarray}
where $p_i^{(n)},q_i^{(n)},s_j^{(n)},l_j^{(n)}\in \mathbb{C}$ for every $i\in \mathbb{Z}$.
By using (\ref{dd10}), (\ref{eee1}) and Lemma \ref{fxc=3} we have
\begin{eqnarray*}
[\phi (L_n),I_m]=-\lambda mI_{m+n},  \
[L_n,\psi(I_m)]=\sum_{j\in \mathbb{Z}}(n-j)s_j^{(m)} L_{j+n}-\sum_{j\in \mathbb{Z}}j l_j^{(m)}I_{n+j}. 
\end{eqnarray*}
This, together with (\ref{aa15}), yields that $l_m^{(m)}=\lambda$ and
\begin{eqnarray*}
&& (n-j)s_j^{(m)}=0, \forall m,n,j, \\
&& j l_j^{(m)}=0, \forall j\neq m, m-n,\\
&& (n-1) \theta_2(I_m)+n\theta_3(I_m)= \rho_1(L_n)=0, \forall m\neq n.
\end{eqnarray*}
From this, it follows by Lemma \ref{refere} that $s_i^{(m)}=l_j^{(m)}=\theta_2(I_m)=\theta_3(I_m)= \rho_1(L_m)=0$ for all $i,j,m\in \mathbb{Z}$
with $j\neq 0, m$. Similarly, by (\ref{aa16}) we deduce that $q_m^{(m)}=\lambda$ and  $p_i^{(m)}=q_j^{(m)}=\rho_2(I_m)=\rho_3(I_m)= \theta_1(L_m)=0$ for all $i,j,m\in \mathbb{Z}$ with $j\neq 0, m$.  Finally,  by (\ref{aa17}) we have that
$$
f(I_n,I_m)=\rho_1(I_n)I_m=\theta_1(I_m)I_n
$$
for all $m,n\in \mathbb{Z}$, which yields that $\rho_1(I_m)=\theta_1(I_m)=0$ for all $m\in \mathbb{Z}$. The proof is completed.
\end{proof}

\begin{lemma}\label{w00}
Let $f$ be a biderivation of $\widetilde{\mathcal{H}}$. Then there are $\lambda \in \mathbb{C}$ and a finite set $\Omega=\{\mu_k\in \mathbb{C}| k\in \mathbb{Z}\}$ such that $ f(x,y)=\lambda[x,y]+r_\Omega(x,y)$
for all $x,y\in \widetilde{\mathcal{H}}$, where $r_\Omega$ is given by Definition \ref{taa}.
\end{lemma}

\begin{proof}
Notice Lemmas \ref{fxc=3}, \ref{fxc=2} and (\ref{aa15}),  (\ref{aa16}), (\ref{aa17}), by a simple computation we have
\begin{eqnarray*}
&& f(L_n,I_m)=\lambda[L_n,I_m]=\lambda[L_n,I_m]+r_\Omega(L_n,I_m), \\
&& f(I_n,L_m)=\lambda[I_n,L_m]=\lambda[I_n,L_m]+r_\Omega(I_n,L_m),\\
&& f(I_n,I_m)=0=\lambda[I_n,I_m]+r_\Omega(I_n,I_m),
\end{eqnarray*}
here we use the fact $r_\Omega(L_n,I_m)=r_\Omega(I_n,L_m)=r_\Omega(I_n,I_m)=0$ by Definition \ref{taa}. This, together with (\ref{aa1888}) and the bilinearity of $f$, completes the proof.
\end{proof}

{\bf The proof of Theorem \ref{posttheo}:}
The ``if" part is easy to verify, we now prove the ``only if" part.

Now we assume that $f$ is a biderivation of $\mathcal{H}$. Define a linear map $\pi: B(\mathcal{H})\rightarrow B(\mathcal{H}/Z^\prime (\mathcal{H}))$ defined by $\pi(f)(x+Z^\prime (\mathcal{H}),y+Z^\prime (\mathcal{H}))=f(x,y)+Z^\prime (\mathcal{H})$ where $Z^\prime (\mathcal{H})={\rm Span} \{C_1,C_2,C_3\}$.  Note that $\mathcal{H}/Z^\prime (\mathcal{H}) \cong \widetilde{\mathcal{H}}$. By Lemmas \ref{fxc=5} and \ref{w00}, we can assume that
$$f(x,y)=\lambda[x,y]+r_\Omega(x,y),\ \forall x,y\in {\rm Span} \{L_i, I_i| i\in \mathbb{Z}\}.$$
where $\lambda \in \mathbb{C}$ and $\Omega=\{\mu_k\in \mathbb{C}| k\in \mathbb{Z}\}$  is a finite set.
For any $x,y\in  {\rm Span} \{L_i, I_i| i\in \mathbb{Z}\}$ and $\alpha, \beta\in {\rm Span} \{C_1,C_2,C_3\}$, we have by Lemma \ref{fxc=4} that
$$f(x+\alpha,y+\beta)= f(x,y)=\lambda[x,y]+r_\Omega(x,y)=\lambda[x+\alpha,y+\beta]+r_\Omega(x+\alpha,y+\beta).$$
The proof is completed.

\section{Applications}\label{sec4}

In this section, we give some applications of biderivations of  the twisted Heisenberg-Virasoro algebra.

\subsection{Linear commuting maps on Lie algebras}

Recall that a linear commuting map $\phi$ on a Lie algebra $L$ subject to $[\phi(x),x]=0$ for any $x\in L$. The first important result on linear (or additive) commuting maps is Posner's theorem \cite{Pos} from 1957. Then many scholars study commuting maps on all kinds of algebra structures. Bre\v{s}ar \cite{Bre3} briefly discuss various extensions of the notion of a commuting map.  About the recent articles on commuting maps we can reference \cite{Bou,Bre2,Bre3,Chen2016,CWS,Hanw,WD1,WD2}.

Obviously, if $\phi$ on
$L$ is such a map, then $[\phi(x), y] = [x, \phi(y)]$ for any $x, y\in L$. Define by $f(x,y)=[\phi(x),y]=[x,\phi(y)]$,
then it is easy to check that $f$ is a biderivation of $L$.

\begin{theorem}
Any linear map $\phi$ on $\mathcal{H}$ is commuting if and only if there are $\lambda \in \mathbb{C}$ and a linear map $\tau : \mathcal{H}\rightarrow Z(\mathcal{H})$ such that  $\phi(x)=\lambda x+\tau (x)$ for all $x\in \mathcal{H}$.
\end{theorem}
\begin{proof}
Notice that the ``if"  part is easy to verify. we now prove the ``only if" part.

By the above discussion, we see that if we let $f(x,y)=[\phi(x),y], \ x, y \in \mathcal{H}$,
 then $f$ is a biderivation of $L$. From Theorem \ref{posttheo}, we have
\begin{eqnarray}
f(x,y)=[\phi(x),y]=\lambda[x,y]+ r_\Omega(x,y). \label{hhh9}
\end{eqnarray}
since $f(x,x)=0$, so $f(x,y)=-f(y,x)$. This implies that $\lambda[x,y]+r_\Omega(x,y)=-\lambda[y,x]-r_\Omega(y,x)$.
Therefore, we see that $r_\Omega(x,y)=r_\Omega(y,x)=-r_\Omega(y,x)$ for all $x,y \in \mathcal{H}$. Thus, $r_\Omega(x,y)=0$.
From this with (\ref{hhh9}), we deduce that $[\phi(x)-\lambda(x), y]=0$ for all $x,y \in \mathcal{H}$.  In other words,
$\phi(x)-\lambda(x) \in Z(\mathcal{H})$.  Hence we can find a linear map $\tau : \mathcal{H}\rightarrow Z(\mathcal{H})$ such that  $\phi(x)=\lambda x+\tau (x)$ for all $x\in \mathcal{H}$.  The proof is completed.
\end{proof}

\subsection{Post-Lie algebra}

Post-Lie algebras have been introduced by Valette in connection with the homology of partition
posets and the study of Koszul operads \cite{vela}. As \cite{Burde1} pointed out, post-Lie algebras are natural common generalization of pre-Lie algebras  and LR-algebras in the geometric context of nil-affine actions of Lie groups. Recently, many authors study some post-Lie algebras and post-Lie algebra structures  \cite{Burde2,Burde1,Mun,pan,tang2014}. In particular, the authors \cite{Burde1} study the commutative post-Lie algebra structure on Lie algebra.  By using our results, we can characterize the  commutative post-Lie algebra structure on $\mathcal{H}$.
Let us recall the following definition of a commutative post-Lie algebra.

\begin{definition}\label{post} \cite{tang2017}
Let $(L, [, ])$ be a complex Lie algebra. A commutative post-Lie algebra structure on $L$ is a
$\mathbb{C}$-bilinear product $x\cdot y$ on $L$ satisfying the following identities:
\begin{eqnarray}
&& x \cdot y = y\cdot x, \nonumber\\
&& [x, y] \cdot z =  x \cdot (y \cdot z)-y \cdot (x \cdot z), \label{post6}\\
&& x\cdot [y, z] = [x\cdot y, z] + [y, x \cdot z] \nonumber
\end{eqnarray}
for all $x, y, z \in L$. We also say $(L, [, ], \cdot)$ a commutative post-Lie algebra.
\end{definition}

A post-Lie algebra $(L, [, ], \cdot)$ is said to be trivial if $x\cdot y=0$ for all $x,y\in L$.
The following lemma shows the connection between commutative post-Lie algebra and biderivation of a Lie algebra.

\begin{lemma}\label{postbide} \cite{tang2017}
Suppose that $(L, [, ], \cdot)$ is a commutative post-Lie algebra. If we define a bilinear map $f : L\times L \rightarrow L$ by $f(x,y)=x\cdot y$ for all $x,y\in L$, then $f$ is a biderivation of $L$.
\end{lemma}

\begin{theorem}\label{refreesttt}
Any commutative post-Lie algebra structure on the twisted Heisenberg-Virasoro algebra $\mathcal{H}$ is trivial.
\end{theorem}

\begin{proof}
Suppose that $(\mathcal{H}, [, ], \cdot)$ is a commutative post-Lie algebra. By Lemma \ref{postbide} and Theorem  \ref{posttheo}, we know that there are $\lambda\in \mathbb{C}$ and a finite set $\Omega=\{\mu_i\in \mathbb{C}|i\in \mathbb{Z}\}$ such that $x\cdot y=\lambda [x,y]+r_\Omega(x,y)$ for all $x,y\in \mathcal{H}$, where $r_\Omega$ is given by Definition \ref{taa}. Because the product $\cdot$ is commutative, we have $\lambda [L_1,L_2]+\chi_\Omega(L_1,L_2)=\lambda [L_2,L_1]+\chi_\Omega(L_2,L_1)$, which implies $\lambda=0$.  By (\ref{post6}), we see that
$$
[L_2,L_1]\cdot L_3=L_2\cdot(L_1\cdot L_3)-L_1\cdot(L_2\cdot L_3).
$$
If there is $\mu_k\in \Omega$ such that $\mu_k\neq 0$, then it is easy to see that the left-hand side of the above equation contains an item $\mu_kI_{6+k}\neq 0$, whereas the right-hand side is equal to zero, which is a contradiction. Thus, we have $\Omega=\{0\}$, i.e., $\mu_i=0$ for any $i\in \mathbb{Z}$. In other words, $r_\Omega=0$. That is, $x\cdot y=0$ for all $x,y\in \mathcal{H}$.
\end{proof}

\subsection{Biderivation of left-symmetric algebras}

Recall that a left-symmetric algebra \cite{Bai,Burde3}  is an algebra $(A,\circ)$ such that $$(x\circ y)\circ z-x\circ (y\circ z)=(y\circ x)\circ z-y\circ (x\circ z)$$
 for all $x,y,z\in A$. Note that an associative algebra is a left-symmetric algebra. A good ting is that a Left-symmetric algebra is Lie admissible, i.e., the commutator $[x, y] = x \circ y - y \circ x$ defines a Lie algebra  $\mathcal{G}(A)$, which is called the sub-adjacent Lie algebra of $A$, and conversely $A$ is called a compatible left-symmetric algebra structure on $\mathcal{G}(A)$. The authors \cite{Chenhj} gave a class of left-symmetric algebra
 structures $A$ on the twisted Heisenberg-Virasoro algebra $\mathcal{H}=\mathcal{G}(A)$ which here is called a graded twisted Heisenberg-Virasoro left-symmetric algebra  as follows.

 \begin{definition}\label{graded}
 A graded twisted Heisenberg-Virasoro left-symmetric algebra is an algebra $(\mathcal{A}, \circ)$ with
 $\mathbb{C}$-basis $\{L_n, I_n, C_1,C_2, C_3| n\in \mathbb{Z} \}$
 such that
 \begin{eqnarray*}
&& L_m\circ L_n=\frac{-n(1+\epsilon n)}{1+\epsilon (m+n)}L_{m+n}+\frac{C_1}{24}(m^3-m+(\epsilon-\epsilon^{-1})m^2)\delta_{m,-n},\\
&& L_m\circ I_n=-n(1+(1-\epsilon n)\alpha \delta_{m,-n})I_{m+n}+ (m^2-m+(\epsilon m^2+m)\beta)C_2 \delta_{m,-n},
 \end{eqnarray*}
\begin{eqnarray*}
&& I_m\circ L_n=n(1+\epsilon n)\alpha \delta_{m,-n})I_{m+n}+ n(1+\epsilon n)\beta)C_2 \delta_{m,-n},\\
&& I_m\circ I_n=\frac{n C_3}{2} \delta_{m,-n},\  C_i\circ C_j=C_i\circ L_m=L_m \circ C_i=C_i\circ I_m=I_m \circ C_i=0
 \end{eqnarray*}
for all $m,n,i,j\in \mathbb{Z}$, where $\alpha, \beta, \epsilon\in \mathbb{C}$ such that
 ${\rm{Re}} \epsilon >0, \epsilon^{-1}\notin \mathbb{Z}$ or ${\rm{Re}} \epsilon=0,  {\rm Im} \epsilon >0.$
 \end{definition}

 The following lemma is easy to verify by a direct computation.

 \begin{lemma}\label{preder}
 Let $(A, \circ)$ be a left-symmetric algebra and $(L, [,])$ be the sub-adjacent Lie algebra of $A$. Then every derivation (resp. biderivation) of $A$ is also a derivation (resp. biderivation) of $L$.
 \end{lemma}

 \begin{theorem}\label{left}
 Any biderivation $f$ of the graded twisted Heisenberg-Virasoro left-symmetric algebra $\mathcal{A}$ is trivial, i.e., $f=0$.
 \end{theorem}

\begin{proof}
Note that the twisted Heisenberg-Virasoro algebra $\mathcal{H}$ is the sub-adjacent Lie algebra of the left-symmetric algebra $\mathcal{A}$ given by Definition \ref{graded}.  Let $f$ be any biderivation of the graded twisted Heisenberg-Virasoro left-symmetric algebra $\mathcal{A}$. Then by Lemma \ref{preder} we see that $f$ is also a biderivation of $\mathcal{H}$. From Theorem  \ref{posttheo}, there are $\lambda\in \mathbb{C}$ and a finite set $\Omega=\{\mu_i\in \mathbb{C}|i\in \mathbb{Z}\}$ such that $f(x, y)=\lambda [x,y]+r_\Omega(x,y)$ for all $x,y\in \mathcal{A}=\mathcal{H}$.
It is not difficult to verify by Definition \ref{graded} that $\lambda=0$ and  $\Omega=0$. Therefore, $f(x,y)=0$ for all $x,y\in \mathcal{A}$.
\end{proof}

\begin{remark}
 As far as we know, there is no any result on biderivation of left-symmetric algebras unless associative algebras. Theorem \ref{left} tells us that any biderivation of the graded twisted Heisenberg-Virasoro left-symmetric algebra is trivial. Is there a non-graded twisted Heisenberg-Virasoro left-symmetric algebra which has
  a non-trivial biderivation? This problem will be interesting.
\end{remark}

\section*{ Acknowledgements}

This work was supported in part by the National Natural Science Foundation of China [grant number 11171294], the Natural Science
Foundation of Heilongjiang Province of China [grant number A2015007], the fund of the Heilongjiang Education Committee [grant number 12531483],
and the special fund of Heilongjiang University of the Fundamental Research Funds for Universities in Heilongjiang province (grant number HDJCCX-2016211).


\begin{thebibliography}{40}
\bibitem{Bai} Bai C. Left-symmetric algebras from linear functions[J]. Journal of algebra, 2004, 281(2): 651-665.
\bibitem{Billig} Billig Y. Representations of the twisted Heisenberg-Virasoro algebra at level zero[J]. Canadian Mathematical Bulletin, 46 (2003), no. 4, 529¨C533.
\bibitem{Bou} Bounds J. Commuting maps over the ring of strictly upper triangular matrices. 2016, Linear Algebra and its Applicaitons 2016 507: 132-136.
\bibitem{Bre1995} Bre\v{s}ar M. On generalized biderivations and related maps. Journal of Algebra, 1995  172(3): 764-786.
\bibitem{Bre2} Bre\v{s}ar M. Miers C. R.  Commuting maps on Lie ideals. Communication in Algebra, 1995 23(14): 5539-5553.
\bibitem{Bre3} Bre\v{s}ar M. Commuting maps: a survey[J]. Taiwanese Journal of Mathematics, 2004, 8(3): pp. 361-397.
\bibitem{Burde2} Burde D, Dekimpe K, Vercammen K. Affine actions on Lie groups and post-Lie algebra structures[J]. Linear Algebra and its Applications, 2012, 437(5): 1250-1263.
\bibitem{Burde1} Burde D, Moens W A. Commutative post-Lie algebra structures on Lie algebras[J]. Journal of algebra, 2016, 467: 183-201.
\bibitem{Burde3} Burde D. Left-symmetric algebras, or pre-Lie algebras in geometry and physics[J]. Central European Journal of Mathematics, 2006, 4(3): 323-357.
\bibitem{Chen2016} Chen Z. Biderivations and linear commuting maps on simple generalized Witt algebras over a field[J]. Electronic Journal of Linear Algebra, 2016, 31(1): 1-12.
\bibitem{Chenhj} Chen H, Li J. Left-symmetric algebra structures on the twisted Heisenberg-Virasoro algebra[J]. Science China Mathematics, 2014, 57(3): 469-476.
\bibitem{Chenhj2} Chen H, Guo X. New simple modules for the Heisenberg-Virasoro algebra[J]. Journal of Algebra, 2013, 390: 77-86.
\bibitem{c2017} Cheng X, Wang M, Sun J,  Zhang, H. Biderivations and linear commuting maps on the Lie algebra $gca$, Linear Multilinear Algebra 2017, DOI: 10.1080/03081087.2016.1277688
\bibitem{CWS} Cheung W. Commuting maps of triangular algebras[J]. Journal of the London Mathematical Society, 2001, 63(1): 117-127.
\bibitem{Du2013} Du Y, Wang Y. Biderivations of generalized matrix algebras[J]. Linear Algebra and its Applications, 2013, 438(11): 4483-4499.
\bibitem{fan2016} Fan G, Dai X. Super-biderivations of Lie superalgebras[J]. Linear and Multilinear Algebra, 2016: 1-9.
\bibitem{Shoulan} Gao S, Jiang C, Pei Y. Low-dimensional cohomology groups of the Lie algebras $W (a, b)$[J]. Communications in Algebra, 2011, 39(2): 397-423.
\bibitem{Gho2013} Ghosseiri N M. On biderivations of upper triangular matrix rings[J]. Linear Algebra and its Applications, 2013, 438(1): 250-260.
\bibitem{Hanw} Han X, Wang D, Xia C. Linear commuting maps and biderivations on the Lie algebras $W (a, b)$[J]. Journal of Lie theory, 2016, 26(3): 777-786.
\bibitem{Mun} Munthe-Kaas H Z, Lundervold A. On post-Lie algebras, Lie-Butcher series and moving frames[J]. Foundations of Computational Mathematics, 2013, 13(4): 583-613.
\bibitem{pan} Pan Y, Liu Q, Bai C, et al. PostLie algebra structures on the Lie algebra $sl(2, \mathbb{C})$[J].  Electronic Journal of Linear Algebra, 2012, 23(1): 23,180-197.
\bibitem{Pos} Posner E C. Derivations in prime rings. P. Am. Math. Soc. 1957, 8(6): 1093-1100.
\bibitem{Shenr} Shen R, Jiang C. The derivation algebra and automorphism group of the twisted Heisenberg-Virasoro algebra[J]. Communications in Algebra, 2006, 34(7): 2547-2558.
\bibitem{tang2016} Tang X.  Biderivations of finite-dimensional complex simple Lie algebras. Linear and Multilinear Algebra, 2017, DOI: 10.1080/03081087.2017.1295433.
\bibitem{tang2017} Tang X. Biderivations, linear commuting maps and commutative post-Lie algebra structures on W-algebras, Communications in Algebra, 2017, dio:10.1080/00927872.2017.1302456
\bibitem{tang2014} Tang X, Zhang Y. Post-Lie algebra structures on solvable Lie algebra $t (2,\mathbb{ C})$[J]. Linear Algebra and its Applications, 2014, 462: 59-87.
\bibitem{vela} Vallette B. Homology of generalized partition posets[J]. Journal of Pure and Applied Algebra, 2007, 208(2): 699-725.
\bibitem{Wangq} Guo H, Wang Q. Twisted Heisenberg-Virasoro vertex operator algebra[J]. arXiv preprint arXiv:1612.06991, 2016.
\bibitem{WD1} Wang D, Yu X. Biderivations and linear commuting maps on the Schr\"{o}dinger-Virasoro Lie algebra[J]. Communications in Algebra, 2013, 41(6): 2166-2173.
\bibitem{WD3} Wang D, Yu X, Chen Z. Biderivations of the parabolic subalgebras of simple Lie algebras[J]. Communications in Algebra, 2011, 39(11): 4097-4104.
\bibitem{WD2} Xia C, Wang D, Han X. Linear super-commuting maps and super-biderivations on the super-Virasoro algebras[J]. Communications in Algebra, 2016, 44(12): 5342-5350.
\end{thebibliography}
\end{document}